\documentclass[
	framed_theorems,
	framed_proofs,
	print,
	english,
]{OCPreprint}

\usepackage{enumitem}

\newcommand\mto{\rightrightarrows}

\begin{document}

\title[Elliptic QVIs under a smallness assumption]{Elliptic quasi-variational inequalities under a smallness assumption: Uniqueness, differential stability and optimal control}
\author[Gerd Wachsmuth]{%
	Gerd Wachsmuth%
}

\maketitle

\begin{abstract}
	We consider a quasi-variational inequality
	governed by a moving set.
	We employ the assumption
	that the movement of the set has a small Lipschitz constant.
	Under this requirement,
	we show that the quasi-variational inequality
	has a unique solution
	which depends Lipschitz-continuously on the source term.
	If the data of the problem is (directionally) differentiable,
	the solution map is directionally differentiable as well.
	We also study the optimal control of the quasi-variational inequality
	and provide necessary optimality conditions of strongly stationary type.
\end{abstract}

\begin{msc}
	\mscLink{47J20},
	\mscLink{49K21},
	\mscLink{35J87}
\end{msc}

\section{Introduction}
We consider the quasi-variational inequality (QVI)
\begin{equation}
	\label{eq:qvi_intro}
	\text{Find } y \in Q(y)
	\quad\text{such that}\quad
	\dual{
		A(y)
		-
		f
	}{v - y}
	\ge
	0
	\qquad
	\forall v \in Q(y)
	.
\end{equation}
Here, $V$ is a Hilbert space,
$A \colon V \to V\dualspace$
is a (possibly nonlinear) mapping,
and $f \in V\dualspace$.
We will not cover the general situation
of a set-valued
mapping $Q \colon V \mto V$,
but
we restrict the treatment of \eqref{eq:qvi_intro}
to the case in which $Q(y)$ is a
\emph{moving set}, i.e.,
\begin{equation}
	\label{eq:moving_set}
	Q(y)
	=
	K + \Phi(y)
\end{equation}
for some non-empty, closed and convex subset $K \subset V$
and $\Phi \colon V \to V$.
It is well-known
that QVIs have
many important real-world applications,
we refer exemplarily to
\cite{BensoussanLions1987,Prigozhin1996,BarrettPrigozhin2013,AlphonseHintermuellerRautenberg2019}
and the references therein.

The main contributions of this paper are the following.
\begin{itemize}
	\item
		We prove existence and uniqueness of solutions to \eqref{eq:qvi_intro}
		under a smallness assumption on the mapping $\Phi$,
		see \cref{sec:QVI_as_VI}.
	\item
		If, additionally, the functions $A$ and $\Phi$ are differentiable
		and if $K$ is polyhedric,
		we establish the directional differentiability of the solution mapping of \eqref{eq:qvi_intro},
		see \cref{sec:diff}.
	\item
		For the associated optimal control problem,
		we derive necessary optimality conditions of strongly stationary type,
		see \cref{sec:optimal_control}.
\end{itemize}
In particular, our results are applicable
if the Lipschitz constant of $\Phi$ is small.
Let us put our work in perspective.
In the following discussion,
we will assume that $A$ is $\mu_A$-strongly monotone
and $L_A$-Lipschitz
and that $\Phi$ is $L_\Phi$-Lipschitz.
We refer to \cref{sec:not} for the definitions.
We further define the condition number of $A$ via $\gamma_A := L_A / \mu_A \ge 1$.
An existence and uniqueness result
for the general QVI \eqref{eq:qvi_intro}
was given in
\cite[Theorem~9]{NoorOettli1994}.
This result can be applied to the moving set case \eqref{eq:moving_set}
via
\cite[Lemma~3.2]{NesterovScrimali2011}.
One obtains the unique solvability of \eqref{eq:qvi_intro}
under the condition
\begin{equation}
	\label{eq:NoorOettli}
	L_\Phi
	<
	1 - \sqrt{1 - 1/\gamma_A^2}
	=
	\frac{1}{\gamma_A \, \paren[\Big]{\gamma_A + \sqrt{\gamma_A^2-1}}}.
\end{equation}
In the work
\cite[Corollary~2]{NesterovScrimali2011}
the requirement was relaxed to
\begin{equation}
	\label{eq:NesterovScrimali}
	L_\Phi
	<
	\frac1{\gamma_A}
	.
\end{equation}
In this work,
we shall show that
\begin{equation}
	\label{eq:my}
	L_\Phi
	<
	\frac{2 \, \sqrt{\gamma_A}}{1 + \gamma_A}
\end{equation}
is sufficient for existence and uniqueness
under the condition that $A$ is the derivative of a convex function.
Note that $A$ is indeed a derivative of a convex function in many important applications.
Moreover, the conditions \eqref{eq:NesterovScrimali} and \eqref{eq:my}
are necessary for uniqueness in the following sense:
Whenever the constants $L_\Phi < 1$, $0 < \mu_A \le L_A$
violate \eqref{eq:NesterovScrimali}
with $\gamma_A := L_A / \mu_A$,
there exist bounded and linear operators
$A \colon V \to V\dualspace$ and $\Phi \colon V \to V$
possessing these constants
such that
\eqref{eq:qvi_intro}
does not have a unique solution for every $f \in V\dualspace$.
If even \eqref{eq:my} is violated,
$A$ can chosen to be symmetric.
We refer to
\cref{thm:sharp_32,thm:sharpness}
below
for the precise formulation of this result.

For a different approach to obtain uniqueness of solutions to \eqref{eq:qvi_intro},
we refer to \cite{Dreves2015}.

To our knowledge,
\cite{AlphonseHintermuellerRautenberg2019}
is the only contribution concerning
differentiability of the solution mapping of \eqref{eq:qvi_intro}.
Their approach is based on monotonicity considerations
and only the differentiability into non-negative directions is obtained.
In what follows,
we are able to relax the assumption required for the differentiability
and we also obtain differentiability in all directions,
see
\cref{thm:differentiability}.

Finally, we are not aware of any contribution in which
stationarity conditions for the optimal control of \eqref{eq:qvi_intro} are obtained.

\section{Notation and preliminaries}
\label{sec:not}
Throughout this work,
$V$ will denote a Hilbert space.
Its dual space is denoted by $V\dualspace$.
The radial cone, the tangent cone
and the normal cone
of a closed, convex set $K \subset V$
at $y \in K$
are given by
\begin{align*}
	\RR_K(y)
	&:=
	\cone(K - y)
	=
	\bigcup_{\alpha > 0} \alpha \, (K-y)
	,
	\qquad
	\TT_K(y)
	:=
	\cl\braces{\RR_K(y)}
	,
	\\
	\TT_K(y)\polar
	&:=
	\set[\big]{\lambda \in V\dualspace \given \dual{\lambda}{v - y} \le 0 \; \forall v \in K},
\end{align*}
respectively.
The critical cone of $K$ w.r.t.\ $(y,\lambda) \in K \times \TT_K(y)\polar$
is given by
\begin{equation*}
	\KK_K(y, \lambda)
	:=
	\TT_K(y) \cap \lambda\anni
	=
	\set[\big]{v \in \TT_K(y) \given \dual{\lambda}{v} = 0 }
	.
\end{equation*}
The set $K$ is called polyhedric at $(y,\lambda)$
if $\KK_K(y,\lambda) = \cl\braces[\big]{\RR_K(y) \cap \lambda\anni}$.
We refer to \cite{Wachsmuth2016:2}
for a recent review of polyhedricity.

A mapping $B\colon V \to V\dualspace$
is called $\mu$-strongly monotone
if $\mu > 0$
satisfies
\begin{equation*}
	\dual{B(y_1) - B(y_2)}{y_1 - y_2}
	\ge
	\mu \, \norm{y_1 - y_2}_V^2
	\qquad
	\forall y_1,y_2 \in V.
\end{equation*}
If $H$ is another Hilbert space,
a mapping $C\colon V \to H$ is called $L$-Lipschitz
for some $L \ge 0$
if
\begin{equation*}
	\norm{C(y_1) - C(y_2)}_H
	\le
	L \, \norm{y_1 - y_2}_V
	\qquad
	\forall y_1,y_2 \in V.
\end{equation*}

The monotonicity of an operator implies
some weak lower semicontinuity.
\begin{lemma}
	\label{lem:wlsc}
	Let $A \colon V \to V\dualspace$ be a monotone operator.
	Suppose that $y_n \weakly y$ in $V$ and $A(y_n) \weakly A(y)$ in $V\dualspace$.
	Then,
	\begin{equation*}
		\liminf_{n \to \infty}
		\dual{A(y_n)}{y_n}
		\ge
		\dual{A(y)}{y}
		.
	\end{equation*}
\end{lemma}
\begin{proof}
	From the monotonicity of $A$ we find
	\begin{align*}
		\dual{A(y_n)}{y_n}
		\ge
		\dual{A(y_n)}{y  }
		+
		\dual{A(y  )}{y_n}
		-
		\dual{A(y  )}{y  }
		.
	\end{align*}
	The right-hand side
	converges towards $\dual{A(y)}{y}$
	due to
	the weak convergences $y_n \weakly y$ and $A(y_n) \weakly A(y)$.
	This implies the claim.
\end{proof}
In the case that $A$ is additionally bounded and linear,
the above claim can be obtained from the observation that
$y \mapsto \dual{A(y)}{y}$ is convex.
This convexity does not hold in the nonlinear setting:
consider $A \colon \R \to \R$, $y \mapsto \max(-1,\min(1,y))$.

In order to obtain unique solvability of \eqref{eq:qvi_intro}
via contraction-type arguments,
one typically requires an inequality like
\begin{equation}
	\label{eq:lipschitz_projection}
	\norm{
		\Proj_{Q(x)}(z) - \Proj_{Q(y)}(z)
	}_V
	\le
	L_Q \, \norm{x - y}_V
	\qquad\forall x,y,z \in V
\end{equation}
for some $L_Q \ge 0$,
see, e.g., \cite[Theorem~4.1]{NesterovScrimali2011}.
Note that this inequality is not related
to the Lipschitz continuity of the projection
since the arguments of the projections in \eqref{eq:lipschitz_projection}
coincide.
By means of an example,
we show that \eqref{eq:lipschitz_projection} does not hold
for obstacle-type problems
if $Q$ is not of the moving-set type.
We consider the setting
$\Omega = (0,1)$,
$V = H_0^1(\Omega)$
$f = 1$, $A = -\Delta$
and
\begin{equation*}
	K(h)
	:=
	\set{
		v \in V
		\given
		v \le h
	}
\end{equation*}
for $\R \ni h \ge 0$.
It is easy to check that
the projection of $A^{-1} f$
onto the set $K(h)$
is given by
\begin{equation*}
	y_h(x)
	=
	\begin{cases}
		t_h \, x - \frac12 \, x^2 & \text{for } x \le t_h \\
		h & \text{for } t_h < x < 1-t_h \\
		-t_h \, x - \frac12 \, x^2 & \text{for } x \ge 1-t_h
	\end{cases}
\end{equation*}
with $t_h := \sqrt{2\,h}$
for all $h \in [0,1/8]$.
Here, we used the norm $\norm{v}_{H_0^1}^2 = \int_\Omega \abs{\nabla v}^2\,\dx$.
Then, $\norm{y_h}_{H_0^1(\Omega)} = C \, h^{3/4}$
for some $C > 0$.
Since $y_0 = 0$,
the mapping $h \mapsto y_h$ is not Lipschitz at $h = 0$.
By choosing a suitable $\Psi \colon V \to \R$
it can be checked that $Q = K \circ \Psi$
violates \eqref{eq:lipschitz_projection}.

\section{Moving-set QVIs as VIs}
\label{sec:QVI_as_VI}
In this section,
we utilize
the moving-set structure of $Q(y)$
to recast the QVI \eqref{eq:qvi_intro}
as an equivalent variational inequality (VI).
This is a classical approach,
see also \cite[Section~2]{AlphonseHintermuellerRautenberg2019},
\cite[Section~5.1]{AlphonseHintermuellerRautenberg2018}.
We start by
defining the new solution variable
$z := y - \Phi(y) \in K$.
In order to not lose any information,
we require that the function $I - \Phi\colon V \to V$
is a bijection.
Hence, $y = (I - \Phi)^{-1}(z)$.
Now,
it is easy to check that \eqref{eq:qvi_intro} is equivalent to
\begin{equation}
	\label{eq:VI}
	\text{Find } z \in K
	\quad\text{such that}\quad
	\dual[\big]{
		(A \circ (I-\Phi)^{-1}) (z)
		-
		f
	}{v - z}
	\ge
	0
	\qquad
	\forall v \in K
\end{equation}
and $y = (I - \Phi)^{-1}(z)$.
This means the following:
Under the assumption that $I-\Phi$ is a bijection,
$y$ is a solution of \eqref{eq:qvi_intro}
if and only if
$z = (I - \Phi)(y)$
is a solution of \eqref{eq:VI}.
In what follows,
we are going to use the VI \eqref{eq:VI}
in order to obtain information about the QVI \eqref{eq:qvi_intro}.
In the case in which both $A$ and $\Phi$
are linear,
such a strategy was suggested in
\cite[Remark~7]{AlphonseHintermuellerRautenberg2019}.
We shall see that this is also viable
in the fully nonlinear case.

In order to analyze \eqref{eq:VI}
we will frequently make use
of the following assumption.
\begin{assumption}
	\label{asm:coercivity}
	The operator $I - \Phi \colon V \to V$
	is a bijection
	and
	the operator
	$B := A \circ (I-\Phi)^{-1}$
	is
	strongly monotone and Lipschitz continuous.
\end{assumption}
Using the equivalence of \eqref{eq:qvi_intro} and \eqref{eq:VI}
as well as the existence result
\cite[Corollary~III.1.8]{KinderlehrerStampacchia1980}
for \eqref{eq:VI},
we obtain the following
existence and uniqueness result for \eqref{eq:qvi_intro}.
\begin{theorem}
	\label{thm:existence_uniqueness}
	Under \cref{asm:coercivity},
	the QVI \eqref{eq:qvi_intro}
	has a unique solution $y \in V$
	for any $f \in V\dualspace$.
	Moreover, the mapping $f \mapsto y$ is
	Lipschitz continuous.
\end{theorem}

In the remainder of this section,
we give some conditions
implying \cref{asm:coercivity}.
\begin{lemma}
	\label{lem:coercivity_A_I_Phi}
	We assume that $A$ is $\mu_A$-strongly monotone
	and $L_A$-Lipschitz
	and that $\Phi$ is $L_\Phi$-Lipschitz.
	We further assume that
	\begin{equation}
		\label{eq:general_L}
		L_\Phi
		<
		\frac1{\gamma_A},
	\end{equation}
	where $\gamma_A = L_A / \mu_A$.
	Then, the operator $B := A \circ (I - \Phi)^{-1}$
	is $\mu_B$-strongly monotone
	and
	$L_B$-Lipschitz
	with
	\begin{equation*}
		\mu_B =
		\frac{\mu_a - L_A \, L_\Phi}{(1 + L_\Phi)^2}
		\qquad\text{and}\qquad
		L_B = \frac{L_A}{1 - L_\Phi}
		.
	\end{equation*}
	In particular,
	\cref{asm:coercivity}
	is satisfied and \eqref{eq:qvi_intro}
	has a unique solution for every $f \in V\dualspace$.
\end{lemma}
\begin{proof}
	First we remark that we have
	$L_\Phi < \gamma_A^{-1} \le 1$.
	Thus,
	Banach's fixed point theorem implies
	that $I-\Phi$ is a bijection.
	We claim that $(1-L_\Phi)^{-1}$
	is a Lipschitz constant of $(I-\Phi)^{-1}$.
	Indeed,
	let $y_1, y_2 \in V$ be arbitrary.
	We define $x_i := (I-\Phi)^{-1}(y_i)$ for $i = 1,2$.
	Then,
	$x_i - y_i = \Phi(x_i)$, $i = 1,2$
	and this yields
	\begin{equation*}
		\norm{x_1 - x_2}_V - \norm{y_1 - y_2}_V
		\le
		\norm{ (x_1 - y_1) - (x_2 - y_2) }_V
		\le
		L_\Phi \, \norm{x_1 - x_2}_V
		.
	\end{equation*}
	This shows the claim concerning a Lipschitz constant
	of $(I-\Phi)^{-1}$.
	Moreover, this directly shows
	that $L_B$ is a Lipschitz constant of $B$.

	For arbitrary $y_1, y_2 \in V$ we again use
	$x_i := (I-\Phi)^{-1}(y_i)$ for $i = 1,2$.
	Then,
	\begin{align*}
		\dual[\big]{B(y_1) - B(y_2)}{y_1 - y_2}
		&=
		\dual[\big]{A (x_1) - A (x_2)}{(I-\Phi)(x_1) - (I-\Phi)(x_2)}
		\\
		&\ge
		(\mu_A - L_A \, L_\Phi) \, \norm{x_1 - x_2}_V^2
		.
	\end{align*}
	The estimate
	\begin{equation*}
		\norm{ y_1 - y_2 }_V
		=
		\norm[\big]{ x_1 - x_2 - \parens[\big]{\Phi(x_1) - \Phi(x_2)} }_V
		\le
		(1 + L_\Phi) \, \norm{x_1 - x_2}_V
	\end{equation*}
	yields the assertion
	concerning the strong monotonicity of $B$.
	The final claim follows from \cref{thm:existence_uniqueness}.
\end{proof}
We recall that the condition \eqref{eq:general_L}
was used in \cite[Corollary~2]{NesterovScrimali2011}
to obtain existence and uniqueness
for solutions of \eqref{eq:VI}.
The above analysis shows,
that this condition even implies \cref{asm:coercivity}.

Next, we show that the estimate \eqref{eq:general_L}
can be significantly relaxed
if $A$ is the derivative of a convex function.
To this end,
we need to recall
an important
inequality
for convex functions.
This inequality is well-known
in the finite-dimensional case,
see, e.g.,
\cite[Theorem~2.1.12]{Nesterov2004}
or
\cite[Lemma~3.10]{Bubeck2015},
and the proof carries over to arbitrary Hilbert spaces.
We are, however,
not aware of a reference in the infinite-dimensional case.
\begin{lemma}
	\label{lem:convex_inequality}
	Let $g\colon V \to \R$
	be a Fréchet differentiable convex function
	such that the derivative $g' \colon V \to V\dualspace$
	is $\mu_g$-strongly monotone
	and $L_g$-Lipschitz.
	Then,
	\begin{equation*}
		\dual{g'(x_1) - g'(x_2)}{x_1 - x_2}
		\ge
		\frac{\mu_g \, L_g}{\mu_g + L_g} \, \norm{x_1 - x_2}_V^2
		+
		\frac{1}{\mu_g + L_g} \, \norm{g'(x_1) - g'(x_2)}_{V\dualspace}^2
	\end{equation*}
	holds for all $x_1, x_2 \in V$.
\end{lemma}
\begin{proof}
	One can transfer the proofs of
	\cite[Theorem~2.1.12]{Nesterov2004}
	or
	\cite[Lemma~3.10]{Bubeck2015}
	to the infinite-dimensional case by using
	\cite[Theorem~18.15]{BauschkeCombettes2011}.
\end{proof}
\begin{lemma}
	\label{lem:coercivity_A_I_Phi_potential}
	We assume that $A$ is $\mu_A$-strongly monotone
	and $L_A$-Lipschitz
	and that $\Phi$ is $L_\Phi$-Lipschitz.
	We further assume that
	there exists
	a Fréchet differentiable convex function
	$g\colon V \to \R$
	such that $A = g'$
	and
	\begin{equation}
		\label{eq:convex_L}
		L_\Phi
		<
		\frac{2 \, \sqrt{\gamma_A}}{1 + \gamma_A}
		=
		\frac{2 \, \sqrt{\mu_A \, L_A}}{\mu_A + L_A}
	\end{equation}
	where $\gamma_A = L_A / \mu_A$.
	Then, the operator $B := A \, (I - \Phi)^{-1}$
	is $\mu_B$-strongly monotone
	and
	$L_B$-Lipschitz
	with
	\begin{equation*}
		\mu_B =
		\frac{4 \, \mu_a \, L_A - L_\Phi^2 \, (\mu_A + L_A)^2}{4 \, (\mu_A + L_A) \, (1+L_\Phi)^2}
		\qquad\text{and}\qquad
		L_B = \frac{L_A}{1 - L_\Phi}
		.
	\end{equation*}
	In particular,
	\cref{asm:coercivity}
	is satisfied and \eqref{eq:qvi_intro}
	has a unique solution for every $f \in V\dualspace$.
\end{lemma}
\begin{proof}
	By arguing as in the proof of \cref{lem:coercivity_A_I_Phi},
	we obtain that $I-\Phi$ is invertible
	and the value of the Lipschitz constant $L_B$ follows.

	Now, let $y_1, y_2 \in V$ be arbitrary
	and we set
	$x_i := (I-\Phi)^{-1}(y_i)$, $i =1,2$.
	Then,
	we apply \cref{lem:convex_inequality}
	to obtain
	\begin{align*}
		\dual[\big]{B(y_1) - B(y_2)}{y_1 - y_2}
		&=
		\dual[\big]{A (x_1) - A (x_2)}{(I-\Phi)(x_1) - (I-\Phi)(x_2)}
		\\
		&\ge
		\frac{\mu_A \, L_A}{\mu_A + L_A} \, \norm{x_1 - x_2}_V^2
		+
		\frac1{\mu_A + L_A} \, \norm{A(x_1) - A(x_2)}_{V\dualspace}^2
		\\&\qquad
		-
		L_\Phi \, \norm{x_1 - x_2}_V \, \norm{A(x_1) - A(x_2)}_{V\dualspace}
		.
	\end{align*}
	Next, we employ Young's inequality
	\begin{align*}
		L_\Phi \, \norm{x_1 - x_2}_V \, \norm{A(x_1) - A(x_2)}_{V\dualspace}
		&\le
		\frac{L_\Phi^2 \, (\mu_A + L_A)}{4} \, \norm{x_1 - x_2}_V^2
		\\&\qquad+
		\frac1{\mu_A + L_A} \, \norm{A(x_1) - A(x_2)}_{V\dualspace}^2
	\end{align*}
	and get
	\begin{align*}
		\dual[\big]{B(y_1) - B(y_2)}{y_1 - y_2}
		&\ge
		\paren[\Big]{\frac{\mu_A \, L_A}{\mu_A + L_A} - \frac{L_\Phi^2 \, (\mu_A + L_A)}{4}} \, \norm{x_1 - x_2}_V^2
		\\
		&=
		\frac{4 \, \mu_A \, L_A - L_\Phi^2 \, (\mu_A + L_A)^2}{4 \, (\mu_A + L_A)} \, \norm{x_1 - x_2}_V^2
		.
	\end{align*}
	From the proof of \cref{lem:coercivity_A_I_Phi}
	we find
	$\norm{ y_1 - y_2 }_V \le (1 + L_\Phi) \, \norm{x_1 - x_2}_V$
	and this yields the monotonicity.
	The final claim follows from \cref{thm:existence_uniqueness}.
\end{proof}
Note that
the inequality
\eqref{eq:convex_L}
is weaker than
\eqref{eq:general_L},
unless $\gamma_A = 1$.
\cref{lem:coercivity_A_I_Phi_potential}
is an improvement
of the corresponding results in the literature,
e.g.,
\cite[Corollary~2]{NesterovScrimali2011},
in the case that $A$ is the derivative of a convex function.
It is well known
that $A$ is a derivative of a convex function
if and only if $A$ is \emph{maximally cyclically monotone},
see, e.g., \cite[Theorem~22.14]{BauschkeCombettes2011}.

Finally, we demonstrate by the mean of two examples
that the assumptions \eqref{eq:general_L} and \eqref{eq:convex_L}
are sharp, even in the case of linear operators.
These examples are found by constructing
operators for which the estimates in the proofs of
\cref{lem:coercivity_A_I_Phi,lem:coercivity_A_I_Phi_potential}
are sharp.
First, we validate the sharpness of \eqref{eq:convex_L}.
\begin{theorem}
	\label{thm:sharp_32}
	Let the constants
	$0 < \mu_A < L_A$
	be given.
	We define
	$L_\Phi := \mu_A / L_A < 1$,
	i.e., \eqref{eq:general_L} is violated.
	Then, there exist linear operators $A$ and $\Phi$ on $V = \R^2$
	(equipped with the Euclidean inner product),
	such that
	$A$
	is $\mu_A$-strongly monotone and $L_A$-Lipschitz,
	$\Phi$ is $L_\Phi$-Lipschitz
	and
	$A \, (I-\Phi)^{-1}$
	is not coercive.
	Moreover,
	there exists a one-dimensional subspace $K \subset \R^2$
	such that \eqref{eq:VI} and \eqref{eq:qvi_intro}
	are not uniquely solvable for all $f \in V\dualspace$.
\end{theorem}
\begin{proof}
	We define the constant $c_A := \sqrt{L_A^2 - \mu_A^2} > 0$
	and the operators
	\begin{equation*}
		A :=
		\begin{pmatrix}
			\mu_A & -c_A \\
			c_A & \mu_A
		\end{pmatrix}
		,
		\qquad
		\Phi :=
		\frac{L_\Phi}{L_A} \, y \, x^\top
	\end{equation*}
	where
	\begin{equation*}
		x =
		\begin{pmatrix}
			1 \\ 0
		\end{pmatrix}
		,
		\qquad
		y =
		\begin{pmatrix}
			\mu_A \\ c_A
		\end{pmatrix}
		=
		A \, x
		.
	\end{equation*}
	Since $A$ is the combination of a rotation and a scaling by $L_A$,
	it is easy to check that
	$z^\top A \, z = \mu_A \, \norm{z}^2$ and $\norm{A \, z} = L_A \, \norm{z}$
	hold for all $z \in \R^2$.
	Moreover, the Lipschitz constant of $\Phi$ is $L_\Phi$.
	However,
	\begin{equation*}
		z^\top A \, (I - \Phi)^{-1} \, z = 0,
		\qquad\text{where}\qquad
		z = (I-\Phi) \, x \ne 0
		.
	\end{equation*}
	Hence,
	$A \, (I-\Phi)^{-1}$
	is not coercive.
	Moreover,
	if we set $K = \Span\set{z}$
	it is clear that \eqref{eq:VI}
	is not uniquely solvable for all $f \in V\dualspace = \R^2$.
	Since $I - \Phi$ is a bijection,
	this implies that \eqref{eq:qvi_intro}
	is not uniquely solvable for all $f \in V\dualspace = \R^2$.
\end{proof}
The next result shows that \eqref{eq:convex_L} is sharp.
\begin{theorem}
	\label{thm:sharpness}
	Let
	$0 < \mu_A < L_A$
	be given.
	We define
	$L_\Phi := 2 \, \sqrt{\mu_A \, L_A} / (\mu_A + L_A) < 1$,
	i.e., \eqref{eq:convex_L} is violated.
	Then, there exists a linear symmetric operator $A$ on $V = \R^2$
	(equipped with the Euclidean inner product)
	and a linear operator $\Phi$ in $\R^2$,
	such that
	$A$
	is $\mu_A$-strongly monotone and $L_A$-Lipschitz,
	$\Phi$ is $L_\Phi$-Lipschitz
	and
	$A \, (I-\Phi)^{-1}$
	is not coercive.
	Moreover,
	there exists a one-dimensional subspace $K \subset \R^2$
	such that \eqref{eq:VI} and \eqref{eq:qvi_intro}
	cannot be uniquely solvable for all $f \in V\dualspace$.
\end{theorem}
\begin{proof}
	We define
	\begin{equation*}
		A :=
		\begin{pmatrix}
			\mu_A & 0 \\
			0 & L_A
		\end{pmatrix}
		.
	\end{equation*}
	It is clear that the operator $A$
	is $\mu_A$-strongly monotone and $L_A$-Lipschitz.
	We further set
	\begin{equation*}
		x :=
		\begin{pmatrix}
			\sqrt{L_A / (\mu_A + L_A)} \\
			\sqrt{\mu_A / (\mu_A + L_A)}
		\end{pmatrix}
		,
		\qquad
		\Phi
		:=
		\frac{2}{(\mu_A + L_A)^2}
		\,
		\begin{pmatrix}
			\mu_A^2 \, L_A & \mu_A \, \sqrt{\mu_A \, L_A} \\
			L_A \, \sqrt{\mu_A \, L_A} & \mu_A^2 \, L_A
		\end{pmatrix}
		.
	\end{equation*}
	It can be checked that
	$\Phi$ is $L_\Phi$-Lipschitz
	and $\norm{x} = 1$.
	However,
	\begin{equation*}
		x^\top A \, (I - \Phi) \, x = 0
		\qquad\text{and}\qquad
		y^\top A \, (I - \Phi)^{-1} \, y = 0
	\end{equation*}
	where
	$y = (I-\Phi) \, x \ne 0$.
	Hence,
	$A \, (I-\Phi)$
	and
	$A \, (I-\Phi)^{-1}$
	are not coercive.
	Moreover,
	if we set $K = \Span\set{y}$
	it is clear that \eqref{eq:VI}
	cannot be uniquely solvable for all $f \in V\dualspace = \R^2$.
	Since $I - \Phi$ is a bijection,
	this implies that \eqref{eq:qvi_intro}
	is not uniquely solvable for all $f \in V\dualspace = \R^2$.
\end{proof}

We further mention that it is also possible
to obtain \cref{asm:coercivity}
in situations in which $\Phi$ is ``not small'',
e.g., if $\Phi = \lambda \, I$ with some $\lambda < 1$,
\cref{asm:coercivity}
follows automatically
if $A$ is strongly monotone and Lipschitz
since
$(I-\Phi)^{-1} = (1-\lambda)^{-1} \, I$
in this case.
Moreover,
it is possible to analyze
the situation in which $A$ is a small perturbation
of the derivative of a convex function
by combining the ideas of \cref{lem:coercivity_A_I_Phi,lem:coercivity_A_I_Phi_potential}.

The combination of \cref{thm:existence_uniqueness,lem:coercivity_A_I_Phi}
yields a well-known result:
under the assumption \eqref{eq:general_L},
the QVI \eqref{eq:qvi_intro} has a unique solution.
Such a result is typically shown via contraction-type arguments,
see, e.g.,
\cite{NesterovScrimali2011}
or
\cite[Section~3.1.1]{AlphonseHintermuellerRautenberg2018}.
Thus, the approach of this section is able to reproduce this classical result.
However, the combination of \cref{thm:existence_uniqueness,lem:coercivity_A_I_Phi_potential}
yields a new result in case that $A$ has a convex potential
in which the condition \eqref{eq:general_L}
on the Lipschitz constant $L_\Phi$ of $\Phi$
is relaxed to
\eqref{eq:convex_L}.

\section{Localization of the smallness assumption}

We localize the assumptions concerning the Lipschitz constant of $\Phi$.
\begin{assumption}
	\label{asm:local_assumption}
	We assume that $A \colon V \to V\dualspace$ is (globally) $\mu_A$-strongly monotone and $L_A$-Lipschitz.
	Further, let $\bar f \in V\dualspace$ be given and let $\bar y$ be a solution of \eqref{eq:qvi_intro}.
	We suppose that there is a closed, convex neighborhood $Y \subset V$ of $\bar y$
	such that $\Phi$ is $L_\Phi$-Lipschitz continuous on $Y$.
	Finally,
	\begin{enumerate}[label=(\roman*)]
		\item
			inequality \eqref{eq:general_L} holds or
		\item
			inequality \eqref{eq:convex_L} holds and $A$ is the Fréchet derivative of a convex function.
	\end{enumerate}
\end{assumption}

\begin{theorem}
	\label{thm:localization}
	Suppose that \cref{asm:local_assumption} is satisfied.
	There is a neighborhood $F \subset V\dualspace$ of $f$
	such that \eqref{eq:qvi_intro} has exactly one solution in $Y$
	for all $f \in F$.
	Moreover, this solution depends Lipschitz-continuously on $f$.
\end{theorem}
Note that we do not claim that \eqref{eq:qvi_intro} is uniquely solvable for all $f \in F$
and
\eqref{eq:qvi_intro} might have further solutions in $V \setminus Y$.
\begin{proof}
	We define $\tilde\Phi \colon V \to V$ via
	\begin{equation*}
		\tilde\Phi(y)
		:=
		\Phi\paren[\big]{\Proj_Y(y)}.
	\end{equation*}
	Since projections are $1$-Lipschitz, $\tilde\Phi$ is $L_\Phi$-Lipschitz.
	Now, we consider the modified QVI
	\begin{equation}
		\label{eq:qvi_mod}
		\text{Find } y \in \tilde Q(y)
		\quad\text{such that}\quad
		\dual{
			A(y)
			-
			f
		}{v - y}
		\ge
		0
		\qquad
		\forall v \in \tilde Q(y)
	\end{equation}
	with
	\begin{equation*}
		\tilde Q(y)
		=
		K + \tilde\Phi(y)
		.
	\end{equation*}
	From
	\cref{asm:local_assumption,lem:coercivity_A_I_Phi,lem:coercivity_A_I_Phi_potential,thm:existence_uniqueness}
	it follows that \eqref{eq:qvi_mod} has a unique solution $y = \tilde S(f)$ for every $f \in F$
	and the solution operator $S \colon V\dualspace \to V$ is Lipschitz continuous.
	Hence, we can choose a neighborhood $F \subset V\dualspace$ of $\bar f$,
	such that $\tilde S(f) \in Y$ for all $f \in F$.

	Since $Q(y) = \tilde Q(y)$ for all $y \in Y$,
	it is clear that $y \in Y$
	is a solution of \eqref{eq:qvi_intro} if and only if
	$y \in Y$ solves \eqref{eq:qvi_mod}.
	Hence, \eqref{eq:qvi_intro} has a unique solution in $Y$
	for all $f \in F$.
\end{proof}

\section{Differential stability}
\label{sec:diff}
In this section,
we consider the situation of \cref{asm:local_assumption}.
However, we do not need \cref{asm:local_assumption}
directly,
but the assertion of \cref{thm:localization} is enough.

\begin{assumption}
	\label{asm:diff}
	We suppose that the following assumptions are satisfied.
	\begin{enumerate}[label=(\roman*)]
		\item\label{item:asm_initial}
			We assume the existence of sets $F \subset V\dualspace$, $Y \subset V$
			such that
			for every $f \in F$,
			\eqref{eq:qvi_intro} has a unique solution $y$ in $Y$
			and the solution map
			$S \colon F \to Y$, $f \mapsto y$
			is Lipschitz continuous.
			For fixed $\bar f \in F$, we set $\bar y := S(\bar f)$.
			The sets $F$, $Y$ are assumed to be neighborhoods
			of $\bar f$, $\bar y$, respectively.

		\item\label{item:asm_Phi}
			The operator $\Phi \colon V \to V$ is
			Lipschitz on $Y$, i.e.,
			there exists $L_\Phi > 0$ with
			\begin{equation*}
				\norm{ \Phi( y_1 ) - \Phi( y_2 ) }_V
				\le
				L_\Phi \,
				\norm{ y_1 - y_2 }_V
				\qquad\forall y_1, y_2 \in Y.
			\end{equation*}
			We suppose that
			$I - \Phi \colon Y \to Z$
			is bijective with a Lipschitz continuous inverse,
			where
			$Z := (I - \Phi)(Y)$.
			Further,
			$\Phi$ is Fréchet differentiable at $\bar y$
			and the bounded linear operator
			$I - \Phi'(\bar y)$ is bijective.

		\item\label{item:asm_A}
			The operator $A$ is Fréchet differentiable at
			$\bar y$
			and the bounded linear operator
			\begin{equation}
				\label{eq:linearized_operator}
				A'\paren{\bar y}
				\,
				(I - \Phi'(\bar y))^{-1}
			\end{equation}
			is assumed to be coercive.

		\item\label{item:asm_polyhedric}
			The set $K$ is polyhedric at $(\bar z, \bar f - A(\bar y))$,
			where $\bar z = (I - \Phi)(\bar y)$.
	\end{enumerate}
\end{assumption}
Due to \eqref{eq:moving_set},
the last assumption is equivalent to the polyhedricity
of $Q(y)$ at $(\bar y, \bar f - A(\bar y))$.

First, we show that \cref{asm:diff} follows from \cref{asm:local_assumption}
and from the differentiability of $\Phi$ and $A$.
\begin{theorem}
	\label{thm:asm51}
	Suppose that \cref{asm:local_assumption}
	is satisfied.
	Then, 
	\cref{asm:diff}~(i)
	holds.
	If $\Phi$ is Fréchet differentiable at $\bar y$,
	then 
	\cref{asm:diff}~\ref{item:asm_Phi} holds.
	If, additionally, $A$ is Fréchet differentiable at $\bar y$,
	then 
	\cref{asm:diff}~\ref{item:asm_A} is satisfied.
\end{theorem}
\begin{proof}
	\cref{asm:diff}~\ref{item:asm_initial}
	follows from \cref{thm:localization}.

	Since $L_\Phi < 1$,
	Banach's fixed point theorem implies that
	$I - \Phi$ is bijective with a Lipschitz continuous inverse.
	The invertibility of $I - \Phi'(\bar y)$
	follows
	from the Neumann series since $\norm{\Phi'(\bar y)} \le L_\Phi < 1$.

	If $A$ is Fréchet differentiable at $\bar y$,
	\cref{asm:local_assumption}
	implies that $A'(\bar y)$ is $\mu_A$-strongly monotone and $L_A$-Lipschitz.
	In case that \eqref{eq:general_L} is satisfied,
	we can invoke \cref{lem:coercivity_A_I_Phi}
	to obtain
	\cref{asm:diff}~\ref{item:asm_A}.
	Otherwise, $A$ is the Fréchet derivative of a convex function.
	Hence, $A'(\bar y)$ is symmetric since it is a second Fréchet derivative,
	see \cite[Theorem~5.1.1]{Cartan1967}.
	Thus, $A'(\bar y)$ is the derivative of the convex function
	$v \mapsto \dual{A'(\bar y)\,v}{v}/2$.
	Therefore, we can invoke \cref{lem:coercivity_A_I_Phi_potential}
	to obtain
	\cref{asm:diff}~\ref{item:asm_A}.
\end{proof}

\begin{lemma}
	\label{lem:differentiability_I_Phi}
	Let us assume that \cref{asm:diff}~\ref{item:asm_initial}--\ref{item:asm_Phi} is satisfied.
	Then, $(I - \Phi)^{-1}$
	is Fréchet differentiable at $\bar z := (I - \Phi)(\bar y)$
	and
	$( (I-\Phi)^{-1} )'(\bar z) = (I-\Phi'(\bar y))^{-1}$.
\end{lemma}
\begin{proof}
	For arbitrary $h \in V$ we have
	\begin{equation*}
		h
		=
		(I - \Phi)
		\bracks[\big]{ (I-\Phi)^{-1} (\bar z + h) - \bar y + \bar y} - \bar z
		.
	\end{equation*}
	Using the Fréchet differentiability of $\Phi$ at $\bar y$
	implies
	\begin{equation*}
		h
		=
		(I - \Phi'(\bar y))
		\bracks[\big]{ (I - \Phi)^{-1}(\bar z + h) - \bar y}
		+
		\oo\paren[\big]{\norm{(I-\Phi)^{-1}(\bar z + h) - \bar y}_V}
	\end{equation*}
	as $\norm{(I-\Phi)^{-1}(\bar z + h) - \bar y}_V \to 0$.
	Finally, using the fact that $(I-\Phi)^{-1}$ is Lipschitz
	implies
	\begin{equation*}
		(I-\Phi'(\bar y))^{-1} h
		=
		(I - \Phi)^{-1}(\bar z + h) - (I - \Phi)^{-1}(\bar z)
		+
		\oo(\norm{h}_V)
		\quad\text{as}\quad \norm{h}_V \to 0
		.
	\end{equation*}
	This shows the claim.
\end{proof}

\begin{lemma}
	\label{lem:strong_monotonicity_B}
	Let us assume that \cref{asm:diff}~\ref{item:asm_initial}--\ref{item:asm_A} is satisfied.
	The operator $B := A \circ (I - \Phi)^{-1}$
	is Fréchet differentiable at $\bar z$
	and its Fréchet derivative
	is given by
	$B'(\bar z)= A'(\bar y) \, (I - \Phi'(\bar y))^{-1}$.
\end{lemma}
\begin{proof}
	Follows from \cref{lem:differentiability_I_Phi} together with a chain rule.
\end{proof}

\begin{theorem}
	\label{thm:differentiability}
	Let us assume that \cref{asm:diff} is satisfied.
	Then, the solution map $S$ is directionally differentiable at $\bar f$
	and the directional derivative $x := S'(\bar f; h)$ in direction $h \in V\dualspace$
	is given by the unique solution of the QVI
	\begin{equation}
		\label{eq:linearized_QVI}
		\text{Find } x \in Q^{\bar y}(x)
		\quad\text{such that}\quad
		\dual{
			A'(\bar y) \, x
			-
			h
		}{v - x}
		\ge
		0
		\qquad
		\forall v \in Q^{\bar y}(x)
		,
	\end{equation}
	where the set-valued mapping $Q^{\bar y} \colon V \mto V$ is given by
	\begin{equation*}
		Q^{\bar y}(x)
		:=
		\KK_K(\bar z, \bar f - A(\bar y)) + \Phi'(\bar y) \, x
		.
	\end{equation*}
\end{theorem}
Note that we have
\begin{equation*}
	\KK_K(\bar z, \bar f - A(\bar y)) = \KK_{Q(\bar y)}(\bar y, \bar f - A(\bar y))
\end{equation*}
due to \eqref{eq:moving_set}.
\begin{proof}
	Let $h \in V\dualspace$ be given.
	There exists $T > 0$ such that $\bar f + t \, h \in F$ for all $t \in [0,T)$.
	For $t \in (0,T)$ we define
	\begin{align*}
		y_t &:= S\paren[\big]{\bar f + t \, h},
		&
		x_t &:= \frac{y_t - \bar y}{t}
		\\
		z_t &:= (I - \Phi)(y_t)
		&
		w_t &:= \frac{z_t - \bar z}{t}
		.
	\end{align*}
	Since $S$ is assumed to be Lipschitz continuous on $F$,
	the difference quotients $\set{ x_t \given t \in (0,T) }$
	are bounded in $V$.
	The Lipschitz continuity of $\Phi$
	implies
	the boundedness of $\set{ w_t \given t \in (0,T) }$ in $V$.

	Since $z_t$ solves the VI \eqref{eq:VI}, i.e.,
	\begin{equation*}
		\text{Find } z \in K
		\quad\text{such that}\quad
		\dual[\big]{
			B(z)
			-
			f
		}{v - z}
		\ge
		0
		\qquad
		\forall v \in K
	\end{equation*}
	with $f := \bar f + t \, h$,
	we can apply \cite[Theorem~2.13]{ChristofWachsmuth2017:3}
	to obtain the convergence of the difference quotients $w_t$.
	Let us check that the assumptions of
	\cite[Theorem~2.13]{ChristofWachsmuth2017:3}
	are satisfied.
	The standing assumption
	\cite[Assumption~2.1]{ChristofWachsmuth2017:3}
	is satisfied in our Hilbert space setting
	with $j = \delta_K$ being the indicator function (in the sense of convex analysis)
	of the set $K$.
	The validity of
	\cite[Assumption~2.2]{ChristofWachsmuth2017:3}
	follows from the Taylor expansion
	\begin{equation*}
		B(z_t)
		=
		B(\bar z + t \, w_t)
		=
		B(\bar z) + t \, B'(\bar z) \, w_t + r(t)
	\end{equation*}
	with $r(t) = \oo(\norm{z_t - \bar z}_V) = \oo(t)$,
	see \cref{lem:strong_monotonicity_B}.
	It remains to check that
	the assumption
	\cite[Theorem~2.13~(ii)]{ChristofWachsmuth2017:3}
	holds:
	\begin{itemize}
		\item
			Since $K$ is assumed to be polyhedric at $\bar z$ w.r.t.\ $\bar f - A(\bar y)$,
			its indicator function $\delta_K$
			is twice epi-differentiable at $(\bar z, \bar f - A(\bar y))$,
			see \cite[Corollary~3.3]{ChristofWachsmuth2017:3}.
			Moreover, its second subderivative is the indicator function
			of the critical cone
			$\KK_K(\bar z, \bar f - A(\bar y)) := \TT_K(\bar z) \cap (\bar f - A(\bar y))\anni$.
		\item
			The weak convergence $w_n \weakly w$ in $V$ implies
			$B'(\bar z) \, w_n \weakly B'(\bar z) \, w$ in $V\dualspace$
			and
			$\liminf_{n \to \infty} \dual{B'(\bar z) w_n}{w_n} \ge \dual{B'(\bar z)\,w}{w}$
			follows from the coercivity of the linear operator
			$B'(\bar z) = A'(\bar y) \, (I - \Phi'(\bar y))^{-1}$,
			see
			\cref{lem:wlsc},
			\cref{asm:diff}~\ref{item:asm_A} and \cref{lem:strong_monotonicity_B}.
	\end{itemize}
	Thus, the application of \cite[Theorem~2.13]{ChristofWachsmuth2017:3}
	yields that all accumulation points $w$ of $w_t$ for $t \searrow 0$
	are solutions of the linearized VI
	\begin{equation}
		\label{eq:linearized_VI}
		\text{Find } w \in \KK_K(\bar z, \bar f - A(\bar y))
		\quad\text{such that}\quad
		\dual[\big]{
			B'(\bar z) \, w
			-
			h
		}{v - w}
		\ge
		0
		\qquad
		\forall v \in \KK_K(\bar z, \bar f - A(\bar y))
		.
	\end{equation}
	Since $B'(\bar z)$ is coercive,
	this linearized VI has a unique solution.
	Hence, the last part of
	\cite[Theorem~2.13]{ChristofWachsmuth2017:3}
	implies $w_t \to w$ as $t \searrow 0$.

	It remains to prove the convergence of $x_t$
	towards the solution of \eqref{eq:linearized_QVI}.
	Using the differentiability of $(I - \Phi)^{-1}$,
	we find
	\begin{align*}
		x_t
		&=
		\frac{y_t - \bar y}{t}
		=
		\frac{ (I-\Phi)^{-1}(z_t) - (I-\Phi)^{-1}(\bar z)}{t}
		\\
		&=
		(I-\Phi'(\bar y))^{-1} \frac{z_t - \bar z}{t} + \frac{\oo(\norm{z_t-\bar z}_V)}{t}
		\\
		&\to
		(I-\Phi'(\bar y))^{-1} \, w
		=:
		x
		.
	\end{align*}
	The change of variables
	$w = (I - \Phi'(\bar y)) \, x$
	shows the equivalence of
	\eqref{eq:linearized_QVI} and \eqref{eq:linearized_VI}.
	Thus, $x$ is the unique solution of \eqref{eq:linearized_QVI}.
\end{proof}

Some remarks concerning \cref{thm:differentiability} are in order.
\begin{remark}\leavevmode
	\label{rem:directional_Phi}
	\begin{enumerate}[label=(\roman*)]
		\item
			The polyhedricity assumption \cref{asm:diff}~\ref{item:asm_polyhedric}
			can be replaced by
			the strong twice epi-differentiability of the indicator function $\delta_K$
			in the sense of \cite[Definition~2.9]{ChristofWachsmuth2017:3}.
			Under this generalized assumption,
			the second epi-derivative of $\delta_K$
			appears as a curvature term in the linearized inequalities
			\eqref{eq:linearized_QVI} and \eqref{eq:linearized_VI}.
			Note that the indicator function of the critical cone
			$\KK_K(\bar z, \bar f - A(\bar y))$,
			which appears implicitly
			in \eqref{eq:linearized_QVI} and \eqref{eq:linearized_VI},
			is just the second epi-derivative of $\delta_K$
			in the case of $K$ being polyhedric.
		\item
			We have derived the differentiability result under the assumption that $\Phi$ is Fréchet differentiable
			at $\bar y$.
			In the notation of \cite{ChristofWachsmuth2017:3},
			this translates to linearity of the operator $A_x$.
			However, the inspection of the proof of \cite[Theorem~2.13]{ChristofWachsmuth2017:3}
			entails that it is possible to replace the Fréchet differentiability of $\Phi$
			by the following set of assumptions:
			\begin{enumerate}[label=(\alph*)]
				\item
					$\Phi$ is Bouligand differentiable at $\bar y$, i.e., there exists
					an operator $\Phi'(\bar y; \cdot) \colon V \to V$ such that
					\begin{equation*}
						\norm{ \Phi(\bar y + h) - \Phi(\bar y) - \Phi'(\bar y; h) }_V = \oo(\norm{h}_V)
						\qquad\text{as } \norm{h}_V \to 0.
					\end{equation*}
				\item
					For every sequence $w_n \weakly w$ in $V$, we assume
					\begin{subequations}
						\label{eq:asm_direc_diff}
						\begin{align}
							\label{eq:asm_direc_diff_1}
							(I - \Phi'(\bar y; \cdot))^{-1}(w_n)
							&\weakly
							(I - \Phi'(\bar y; \cdot))^{-1}(w)
							,
							\\
							\label{eq:asm_direc_diff_2}
							\mspace{-32mu}
							\liminf_{n \to \infty} \dual{ A'(\bar y) \, (I - \Phi'(\bar y; \cdot))^{-1}(w_n) }{ w_n }
							&\ge
							\dual{ A'(\bar y) \, (I - \Phi'(\bar y; \cdot))^{-1}(w) }{ w }
							.
						\end{align}
					\end{subequations}
			\end{enumerate}
			Note that (a) implies that
			$\Phi'(\bar y; \cdot)$ is Lipschitz on $V$ with constant $L_\Phi$.
			Hence, $(I - \Phi'(\bar y;\cdot))$ is invertible
			and \cref{lem:coercivity_A_I_Phi,lem:coercivity_A_I_Phi_potential}
			can be used to obtain the strong monotonicity of
			$A'(\bar y) \, (I - \Phi'(\bar y; \cdot))^{-1}$.

			Property \eqref{eq:asm_direc_diff_1} can be verified by assuming, e.g.,
			weak continuity of $\Phi'(\bar y; \cdot)$.
			Indeed,
			the sequence $z_n := (I - \Phi'(\bar y; \cdot))^{-1}(w_n)$
			is bounded, hence, $z_n \weakly z$ along a subsequence.
			Now, weak continuity implies
			$w_n = z_n - \Phi'(\bar y; z_n) \weakly z - \Phi'(\bar y; z)$
			and $w_n \weakly w$ implies $z = (I - \Phi'(\bar y; \cdot))^{-1}(w)$,
			i.e. $z_n \weakly (I - \Phi'(\bar y; \cdot))^{-1}(w)$ along a subsequence.
			The uniqueness of the limit point implies the convergence of the entire sequence.

			Finally, \eqref{eq:asm_direc_diff_2}
			can be obtained via \eqref{eq:asm_direc_diff_1} and \cref{lem:wlsc}.
	\end{enumerate}
\end{remark}

In the next remark,
we compare
our differentiability result
with \cite[Theorem~1]{AlphonseHintermuellerRautenberg2019}.
\begin{remark}
	\label{rem:comparison}
	In
	\cite[Theorem~1]{AlphonseHintermuellerRautenberg2019}
	a similar differentiability result
	is obtained
	in a more restrictive setting:
	\begin{enumerate}[label=(\roman*)]
		\item
			Therein, the leading operator $A$ has to be linear
			and $T$-monotone (w.r.t.\ a vector space order on $V$).
			Our approach also allows for non-linear operators
			and we do not need any order structure on $V$.
			Similarly, we do not need any monotonicity assumptions on $\Phi$.
		\item
			They require the complete continuity of
			$\Phi'(\bar y)$,
			which is not needed in \cref{thm:differentiability}.
		\item
			One of their most restrictive assumptions is
			the assumption (\textbf{A5}).
			Via \cite[Theorem~3.1.2]{Cartan1967},
			this assumption is equivalent to $\Phi$
			being $L_\Phi$-Lipschitz with
			\begin{equation}
				\label{eq:ineq_AlphonseHintermuellerRautenberg2019}
				L_\Phi
				<
				\frac{1}{1 + \gamma_A}.
			\end{equation}
			This inequality is much stronger than \eqref{eq:general_L}.
			Thus, their assumption (\textbf{A5})
			implies that the solutions to the QVI \eqref{eq:qvi_intro}
			are unique.
	\end{enumerate}

	Moreover,
	they obtained the differentiability only for non-negative
	directions
	whereas
	our approach is applicable to arbitrary perturbations of the right-hand side.

	One assumption in \cite{AlphonseHintermuellerRautenberg2019}
	is weaker:
	they only need Hadamard differentiability of $\Phi$.
	We need Fréchet differentiability
	(or Bouligand differentiability, see \cref{rem:directional_Phi}).
\end{remark}

\section{Optimal control}
\label{sec:optimal_control}
In this section,
we consider the optimal control problem
\begin{equation}
	\label{eq:optimal_control}
	\begin{aligned}
		\text{Minimize} \quad & J(y,u) \\
		\text{w.r.t.}\quad & y \in V, u \in U \\
		\text{s.t.}\quad & y \in Q(y)
		\quad\text{and}\quad
		\dual{
			A(y)
			-
			(B \, u + f)
		}{v - y}
		\ge
		0
		\quad
		\forall v \in Q(y)
		.
	\end{aligned}
\end{equation}
Here, $f \in V\dualspace$ is fixed,
$U$ is a Hilbert space
and the bounded, linear operator
$B \colon U \to V\dualspace$
is assumed to have a dense range.
Moreover,
the objective $J \colon V \times U \to \R$
is Fréchet differentiable.

\begin{theorem}
	\label{thm:strong_stationarity}
	Suppose that $(\bar y, \bar u)$ is locally optimal for \eqref{eq:optimal_control}.
	In addition to the assumptions on $B$ and $J$,
	we assume that \cref{asm:diff} is satisfied at $\bar f := B \, \bar u + f$.
	Then, there exist unique multipliers $p \in V$, $\mu \in V\dualspace$
	such that the system
	\begin{subequations}
		\label{eq:optimal_control_VI_OS}
		\begin{align}
			\label{eq:optimal_control_VI_OS_1}
			J_y(\cdot) + A'(\bar y)^\star \, p + (I - \Phi'(\bar y))\adjoint\mu &= 0, \\
			\label{eq:optimal_control_VI_OS_2}
			J_u(\cdot) -B^\star \, p &= 0, \\
			\label{eq:optimal_control_VI_OS_3}
			p &\in -\KK_K(\bar z, \bar \lambda), \\
			\label{eq:optimal_control_VI_OS_4}
			\mu &\in \KK_K(\bar z, \bar \lambda)\polar
		\end{align}
	\end{subequations}
	is satisfied.
	Here,
	\begin{equation}
		\label{eq:barzbarlambda}
		\bar z = (I - \Phi)(\bar y) \in K
		\qquad\text{and}\qquad
		\bar\lambda = B \, \bar u + f - A(\bar y) \in \TT_K(\bar z)\polar,
	\end{equation}
	and
	$J_y(\cdot) \in V\dualspace$ and $J_u(\cdot) \in U\dualspace$
	are the partial Fréchet derivatives of $J$ at $(\bar u, \bar y)$.
\end{theorem}
\begin{proof}
	We use classical arguments
	dating back to \cite[Proposition~4.1]{Mignot1976},
	see also \cite[Theorem~5.3]{Wachsmuth2016:2}.

	Due to \cref{asm:diff} we can invoke \cref{thm:differentiability}
	to obtain the directional differentiability of the
	control-to-state map.
	Combined with the local optimality of $(\bar y, \bar u)$,
	this implies
	\begin{equation}
		\label{eq:B_stat}
		\dual{J_y(\cdot)}{S'(B\,\bar u + f;  B \, h) }_{V\dualspace,V}
		+
		\dual{J_u(\cdot)}{h}_{U\dualspace,U}
		\ge
		0
		\qquad\forall h \in U.
	\end{equation}
	Due to the Lipschitz estimate
	$\norm{S'(B\,\bar u + f;  B \, h)}_{V} \le C \, \norm{B \, h}_{V\dualspace}$,
	the above inequality implies
	\begin{equation*}
		\abs{\dual{J_u(\cdot)}{h}_{U\dualspace,U}}
		\le
		C \, \norm{B \, h}_{V\dualspace}
		\qquad\forall h \in U.
	\end{equation*}

	Hence, there is $p \in V\bidualspace \cong V$
	(by defining it as in the next line on the dense subspace $\image(B) \subset V\dualspace$ and extending it by continuity on the whole space $V\dualspace$)
	such that
	\begin{equation*}
		\dual{J_u(\cdot)}{h}_{U\dualspace,U}
		=
		\dual{p}{B \, h}_{V,V\dualspace}
		\qquad\forall h \in U.
	\end{equation*}
	This yields \eqref{eq:optimal_control_VI_OS_2}
	and
	\begin{equation*}
		\dual{J_y(\cdot)}{S'(B\,\bar u + f;  B \, h)}_{V\dualspace,V}
		+
		\dual{p}{B \, h}_{V,V\dualspace}
		\ge
		0
		\qquad\forall h \in U
		.
	\end{equation*}
	Using the density of $\image(B)$ in $V\dualspace$ we get
	\begin{equation*}
		\tag{$*$}
		\label{eq:opt_sys}
		\dual{J_y(\cdot)}{S'(B\,\bar u + f; h)}_{V\dualspace,V}
		+
		\dual{p}{h}_{V,V\dualspace}
		\ge
		0
		\qquad\forall h \in V\dualspace.
	\end{equation*}

	In what follows, we set $\KK := \KK_K(\bar z, \bar\lambda)$ for convenience.
	We recall that $S'(B \, \bar u + f; h)$
	is the unique solution
	of
	\begin{equation}
		\tag{$**$}
		\label{eq:two_star}
		\text{Find } x \in Q^{\bar y}(x)
		\quad\text{such that}\quad
		\dual{
			A'(\bar y) \, x
			-
			h
		}{v - x}
		\ge
		0
		\qquad
		\forall v \in Q^{\bar y}(x)
		,
	\end{equation}
	where the set-valued mapping $Q^{\bar y} \colon V \mto V$ is given by
	\begin{equation*}
		Q^{\bar y}(x)
		=
		\KK + \Phi'(\bar y) \, x
		.
	\end{equation*}

	We choose $h \in \KK^\circ$ in \eqref{eq:opt_sys}.
	Then, \eqref{eq:two_star} shows $S'(\bar u + f; h) = 0$ and, thus,
	\begin{equation*}
		\dual{p}{h}_{V,V\dualspace}
		\ge
		0
		\qquad\forall h \in \KK^\circ,
	\end{equation*}
	i.e., $p \in -\KK$, which shows \eqref{eq:optimal_control_VI_OS_3}.

	Now, we choose
	$v \in (I - \Phi'(\bar y))^{-1} \KK$
	and set $h = A'(\bar y) \, v$.
	It can be checked that \eqref{eq:two_star}
	implies
	$S'(B \, \bar u + f ; h) = v$.
	With this choice,
	\eqref{eq:opt_sys} implies
	\begin{equation*}
		\dual{J_y(\cdot) + A'(\bar y)^\star \, p}{v}_{V\dualspace,V}
		\ge
		0
		\qquad\forall v \in (I - \Phi'(\bar y))^{-1} \KK.
	\end{equation*}
	We define $\mu := -(I - \Phi'(\bar y))^{-\adjointsign}(J_y(\cdot) + A^\star \, p)$
	and get \eqref{eq:optimal_control_VI_OS_1} and
	\begin{equation*}
		\dual{(I - \Phi'(\bar y))\adjoint \mu}{v}_{V\dualspace,V}
		\le
		0
		\qquad\forall v \in (I - \Phi'(\bar y))^{-1} \KK.
	\end{equation*}
	Since $I - \Phi'(\bar y)$ is a bijection,
	this is equivalent to
	\eqref{eq:optimal_control_VI_OS_4}.

	The uniqueness of $p$ and $\mu$
	follows from the injectivity of $B\adjoint$
	and the bijectivity of
	$(I - \Phi'(\bar y))\adjoint$.
\end{proof}
The approach of
\cite[Section~6.1]{Christof2018}
can be used to provide strong stationarity systems
under less restrictive assumptions on $K$,
i.e., the polyhedricity assumption
can be replaced by the twice 
epi-differentiability of the indicator function $\delta_K$.

In the case that $\Phi$ is merely Bouligand differentiable, cf.\ \cref{rem:directional_Phi},
conditions \eqref{eq:optimal_control_VI_OS_1} and \eqref{eq:optimal_control_VI_OS_4}
could be rewritten as
\begin{equation}
	\label{eq:alternative}
	J_y(\cdot) + A'(\bar y)^\star \, p + \hat\mu = 0,
	\qquad
	\hat\mu \in \bracks[\big]{(I - \Phi'(\bar y;\cdot))^{-1} \KK_K(\bar z, \bar \lambda)}\polar,
\end{equation}
see the proof of \cref{thm:strong_stationarity}.

Finally,
we show that the system of strong stationarity
is of reasonable strength,
i.e., it implies the B-stationarity
\eqref{eq:B_stat}.
\begin{lemma}
	\label{lem:strong_yields_B}
	Let $(\bar y, \bar u)$ be a feasible point of \eqref{eq:optimal_control}
	such that \cref{asm:diff} is satisfied at $\bar f := B \, \bar u + f$.
	Moreover, suppose that $J$ is Fréchet differentiable.
	If there exist multipliers
	$p \in V$, $\mu \in V\dualspace$
	satisfying \eqref{eq:optimal_control_VI_OS},
	then \eqref{eq:B_stat} holds.
\end{lemma}
\begin{proof}
	For an arbitrary $h \in U$ we define
	$x := S'(B \, \bar u + f; B \, h)$.
	Then
	\begin{align*}
		\dual{J_y(\cdot)}{x}_{V\dualspace,V}
		+
		\dual{J_u(\cdot)}{h}_{U\dualspace,U}
		&=
		\dual{-A'(\bar y)^\star \, p - (I - \Phi'(\bar y))\adjoint\mu}{x}_{V\dualspace,V}
		+
		\dual{B\adjoint p}{h}_{U\dualspace,U}
		\\
		&=
		-\dual{p}{A'(\bar y) \, x - B \, h}_{V\dualspace,V}
		-\dual{\mu}{(I - \Phi'(\bar y)) \, x}_{V\dualspace,V}
		.
	\end{align*}
	From the linearized QVI \eqref{eq:linearized_QVI}
	and the strong stationarity system \eqref{eq:optimal_control_VI_OS},
	we have
	\begin{align*}
		(I - \Phi'(\bar y)) \, x &\in \KK_K(\bar z, \bar\lambda),
		&
		A'(\bar y) \, x - B \, h &\in - \KK_K(\bar z, \bar\lambda)\polar,
		\\
		p &\in -\KK_K(\bar z, \bar \lambda),
		&
		\mu &\in \KK_K(\bar z, \bar \lambda)\polar,
	\end{align*}
	where we used \eqref{eq:barzbarlambda}.
	Thus, \eqref{eq:B_stat} follows.
\end{proof}

\subsection*{Acknowledgement}
This work is supported by the DFG grant
\emph{Analysis and Solution Methods for Bilevel Optimal Control Problems}
(grant number WA3636/4-1)
within the Priority Program SPP 1962
(Non-smooth and Complementarity-based Distributed Parameter Systems: Simulation and Hierarchical Optimization).

The simple proof of \cref{lem:wlsc}
was pointed out by Felix Harder.

\printbibliography

\end{document}